\documentclass[reqno]{amsart}
\usepackage{amsmath, amsthm, amssymb, amstext}

\usepackage{hyperref,xcolor}% http://ctan.org/pkg/{hyperref,xcolor}
\hypersetup{
 pdfborder={0 0 0},
 colorlinks,
}
\usepackage{enumitem}
\newtheorem{theorem}{Theorem}
\newtheorem{remark}[theorem]{Remark}

\newtheorem{proposition}[theorem]{Proposition}

\newtheorem{example}[theorem]{Example}
\DeclareMathOperator*{\divergenz}{div}              %
\DeclareMathOperator*{\ints}{int}         %
\DeclareMathOperator*{\Ss}{S}

\newcommand{\R}{\mathbb{R}}

\newcommand{\Lp}[1]{L^{#1}(\Omega)}

\newcommand{\Wp}[1]{W^{1,#1}(\Omega)}

\newcommand{\Wpzero}[1]{W^{1,#1}_0(\Omega)}

\newcommand{\eps}{\varepsilon}
\newcommand{\ph}{\varphi}

\newcommand{\into}{\int_{\Omega}}

\newcommand{\Linf}{L^{\infty}(\Omega)}
\newcommand{\close}{\overline{\Omega}}
\newcommand{\interior}{\ints \left(C^1_0(\overline{\Omega})_+\right)}

\numberwithin{theorem}{section}
\numberwithin{equation}{section}

\title[Constant sign solutions for double phase problems]{Constant sign solutions for double phase problems with superlinear nonlinearity}

\author[L.\,Gasi\'nski]{Leszek Gasi\'nski}
\address[L.\,Gasi\'nski]{Pedagogical University of Cracow, Department of Mathematics, Podchorazych 2, 30-084 Cracow, Poland}
\email{leszek.gasinski@up.krakow.pl}

\author[P.\,Winkert]{Patrick Winkert}
\address[P.\,Winkert]{Technische Universit\"{a}t Berlin, Institut f\"{u}r Mathematik, Stra\ss e des 17.\,Juni 136, 10623 Berlin, Germany}
\email{winkert@math.tu-berlin.de}


\begin{document}

\begin{abstract}
        We study parametric double phase problems involving superlinear nonlinearities with a growth that need not necessarily be polynomial. Based on truncation and comparison methods the existence of two constant sign solutions is shown provided the parameter is larger than the first eigenvalue of the $p$-Laplacian. As a result of independent interest we prove a priori estimates for solutions for a general class of double phase problems with convection term.
\end{abstract}
	
\maketitle
	
%***********************************************************************************************************************************
\section{Introduction}%\label{section_1}
%***********************************************************************************************************************************

Given a bounded domain $\Omega \subseteq \R^N$, $N\geq 2$, with a $C^2$-boundary $\partial \Omega$, we consider the following double phase problem
\begin{equation}\label{problem}
    \begin{aligned}
	-\divergenz\left(|\nabla u|^{p-2}\nabla u+\mu(x) |\nabla u|^{q-2}\nabla u\right) & =\lambda |u|^{p-2}u-f(x,u)\quad && \text{in } \Omega,\\
	u & = 0 &&\text{on } \partial \Omega,
    \end{aligned}
\end{equation}
where $1<p<q<N$, $\lambda>0$ is a parameter specified later, the function $\mu \colon\close \to [0,\infty)$ is supposed to be Lipschitz continuous and $f\colon\Omega\times\R\to \R$ is a Carath\'{e}odory function, that is, $x\mapsto f(x,s)$ is measurable for all $s\in\R$ and $s\mapsto f(x,s)$ is continuous for almost all (a.\,a.) $x\in\Omega$.

The main goal of this paper is an existence result for problem \eqref{problem} which provides two constant sign solutions, one positive and the other negative. The novelty of our work is the fact that the nonlinearity $f\colon\Omega\times\R\to \R$  is only $(p-1)$-superlinear at infinity, $(p-1)$-sublinear near zero and bounded on bounded sets. No polynomial growth condition or any monotonicity condition is needed in our proof which is in contrast to other works in this direction, see for example the recent paper of Gasi\'nski-Papageorgiou \cite[Proposition 3.4]{Gasinski-Papageorgiou-2019}. In addition, we present a priori bounds for weak solutions of problem \eqref{problem} which are based on the recent work of Marino-Winkert \cite{Marino-Winkert-2019} by applying Moser's iteration.

Problems of type \eqref{problem} have an interesting phenomena because the operator on the left-hand side is the so-called double phase operator whose behavior switches between two different elliptic situations depending on the values of the weight function $\mu\colon\close\to [0,\infty)$. In other words, the behavior of the operator is controlled by the sets $\{x\in \Omega: \mu(x)=0\}$ and $\{x\in \Omega: \mu(x) \neq 0\}$.

Zhikov was one of the first who introduced such classes of operators to describe models of strongly anisotropic materials, see \cite{Zhikov-1986}, \cite{Zhikov-1995}, \cite{Zhikov-1997}  and also the monograph of Zhikov-Kozlov-Oleinik \cite{Zhikov-Kozlov-Oleinik-1994}. The main idea was the introduction of the functional
\begin{align}\label{integral_minimizer}
    u \mapsto \int \left(|\nabla u|^p+\mu(x)|\nabla u|^q\right)\,dx,
\end{align}
in order to describe such phenomena. Such functionals have been intensively studied in the past decade. We refer to the papers of Baroni-Colombo-Mingione \cite{Baroni-Colombo-Mingione-2015}, \cite{Baroni-Colombo-Mingione-2016}, \cite{Baroni-Colombo-Mingione-2018}, Baroni-Kussi-Mingione \cite{Baroni-Kussi-Mingione-2015}, Colombo-Mingione \cite{Colombo-Mingione-2015a}, \cite{Colombo-Mingione-2015b} and the references therein concerning the regularity. We also point out that the integrals of the form \eqref{integral_minimizer} arise in the context of functionals with non-standard growth; see the works of Cupini-Marcellini-Mascolo \cite{Cupini-Marcellini-Mascolo-2015} and Marcellini \cite{Marcellini-1989}, \cite{Marcellini-1991}.

The existence of solutions for classes of problem \eqref{problem} has only been studied by few authors. Perera-Squassina \cite{Perera-Squassina-2019} proved the existence of a solution of problem \eqref{problem} by applying Morse theory where they used a cohomological local splitting to get an estimate of the critical groups at zero. The corresponding eigenvalue problem of the double phase operator with Dirichlet boundary condition was treated by Colasuonno-Squassina \cite{Colasuonno-Squassina-2016}, who proved the existence and properties of related variational eigenvalues. By applying variational methods, Liu-Dai \cite{Liu-Dai-2018} treated double phase problems and proved existence and multiplicity results, also sign-changing solutions. A similar treatment has been recently done by Gasi\'nski-Papageorgiou \cite[Proposition 3.4]{Gasinski-Papageorgiou-2019} via the Nehari manifold method. We point out that the proof for the constant sign solutions in \cite{Gasinski-Papageorgiou-2019} needs an additional monotonicity condition which can be avoided in our work. Furthermore, we refer to a recent work of the authors \cite{Gasinski-Winkert-2019} which shows the existence of at least one solution of problems of type \eqref{problem} by applying the surjectivity result for pseudomonotone operators. This can be realized by an easy condition on the convection term, in addition to the usual growth condition. Finally we refer to works which are very related to our topic dealing with types of double phase problems. We mention, for example, Bahrouni-R\u{a}dulescu-Repov\v{s} \cite{Bahrouni-Radulescu-Repovs-2019}, Cencelj-R\u{a}dulescu-Repov\v{s} \cite{Cencelj-Radulescu-Repovs-2018}, Papageorgiou-R\u{a}dulescu-Repov\v{s} \cite{Papageorgiou-Radulescu-Repovs-2019a}, \cite{Papageorgiou-Radulescu-Repovs-2019b}, Papageorgiou-Scapellato \cite{Papageorgiou-Scapellato-2020}, Zhang-R\u{a}dulescu \cite{Zhang-Radulescu-2018} and the references therein. We also refer to the interesting overview article of R\u{a}dulescu \cite{Radulescu-2019} concerning isotropic and anisotropic double phase problems.

Our results are mainly based on truncation and comparison methods combined with the representation of the first eigenvalue of the $p$-Laplacian with homogeneous Dirichlet boundary condition, see \eqref{eigenvalue_problem} in Section \ref{section_2}. Moreover, the $L^\infty$-bounds for weak solutions stated in Section \ref{section_3} are based on Moser's iteration and give a priori bounds for a very general setting of problem \eqref{problem} including convection term, see \eqref{problem2}.

%***********************************************************************************************************************************
\section{Preliminaries}\label{section_2}
%***********************************************************************************************************************************

By $\Lp{r}$ and $L^r(\Omega;\R^N)$ we denote the usual Lebesgue spaces endowed with the norm denoted by $\|\cdot\|_r$ for $1\leq r\leq\infty$ while $\Wp{r}$ and $\Wpzero{r}$ are the Sobolev spaces equipped with the norms denoted by $\|\cdot\|_{1,r}$ and $\|\cdot\|_{1,r,0}$, respectively, for $1\leq r\leq\infty$.

Let $1<p<q<N$ and suppose
\begin{align}\label{condition_poincare}
    \frac{q}{p}<1+\frac{1}{N}, \qquad \mu\colon \close\to [0,\infty) \text{ is Lipschitz continuous}.
\end{align}
Let $\mathcal{H}\colon \Omega \times [0,\infty)\to [0,\infty)$ be the function
\begin{align*}
    (x,t)\mapsto t^p+\mu(x)t^q.
\end{align*}
Then, the Musielak-Orlicz space $L^\mathcal{H}(\Omega)$ is defined by
\begin{align*}
    L^\mathcal{H}(\Omega)=\left \{u ~ \Big | ~ u: \Omega \to \R \text{ is measurable and } \rho_{\mathcal{H}}(u):=\into \mathcal{H}(x,|u|)\,dx< +\infty \right \}
\end{align*}
equipped with the Luxemburg norm
\begin{align*}
    \|u\|_{\mathcal{H}} = \inf \left \{ \tau >0 : \rho_{\mathcal{H}}\left(\frac{u}{\tau}\right) \leq 1  \right \}.
\end{align*}
With this norm, it is clear that the space $L^\mathcal{H}(\Omega)$ is uniformly convex and so reflexive. Moreover, we introduce the seminormed space
\begin{align*}
    L^q_\mu(\Omega)=\left \{u ~ \Big | ~ u: \Omega \to \R \text{ is measurable and } \into \mu(x) |u|^q \,dx< +\infty \right \}
\end{align*}
endowed with the seminorm
\begin{align*}
    \|u\|_{q,\mu} = \left(\into \mu(x) |u|^q \,dx \right)^{\frac{1}{q}}.
\end{align*}
We know from Colasuonno-Squassina \cite[Proposition 2.15 (i), (iv) and (v)]{Colasuonno-Squassina-2016} that the embeddings
\begin{align*}
    \Lp{q} \hookrightarrow L^\mathcal{H}(\Omega) \hookrightarrow \Lp{p} \cap L^q_\mu(\Omega)
\end{align*}
are continuous. A simple calculation shows that
\begin{align}\label{inequality_lp}
    \min \left\{\|u\|_\mathcal{H}^p,\|u\|_\mathcal{H}^q \right\} \leq \|u\|_p^p+\|u\|^q_{q,\mu}
    \leq \max\left \{\|u\|_\mathcal{H}^p,\|u\|_\mathcal{H}^q\right\}
\end{align}
for all $u\in L^\mathcal{H}(\Omega)$. By $W^{1,\mathcal{H}}(\Omega)$ we denote the corresponding Sobolev space which is defined by
\begin{align*}
    W^{1,\mathcal{H}}(\Omega)= \left \{u \in L^\mathcal{H}(\Omega) : |\nabla u| \in L^{\mathcal{H}}(\Omega) \right\}
\end{align*}
equipped with the norm
\begin{align*}
    \|u\|_{1,\mathcal{H}}= \|\nabla u \|_{\mathcal{H}}+\|u\|_{\mathcal{H}},
\end{align*}
where $\|\nabla u\|_\mathcal{H}=\||\nabla u|\|_{\mathcal{H}}$.

The completion of $C^\infty_0(\Omega)$ in $W^{1,\mathcal{H}}(\Omega)$ is denoted by $W^{1,\mathcal{H}}_0(\Omega)$ and due to \eqref{condition_poincare} we have an equivalent norm on $W^{1,\mathcal{H}}_0(\Omega)$ given by
\begin{align*}
    \|u\|_{1,\mathcal{H},0}=\|\nabla u\|_{\mathcal{H}},
\end{align*}
see Proposition 2.18 in Colasuonno-Squassina \cite{Colasuonno-Squassina-2016}. Since both spaces $W^{1,\mathcal{H}}(\Omega)$ and $W^{1,\mathcal{H}}_0(\Omega)$ are uniformly convex, we know that they are reflexive Banach spaces.

From Colasuonno-Squassina \cite[Proposition 2.15]{Colasuonno-Squassina-2016} we have the compact embedding
\begin{align*}%\label{embedding_sobolev}
    W^{1,\mathcal{H}}_0(\Omega) \hookrightarrow \Lp{r}
\end{align*}
for each $1<r<p^*$, where $p^*$ is the critical exponent to $p$ given by
\begin{align}\label{critical_exponent}
    p^*:=\frac{Np}{N-p}.
\end{align}
A direct consequence of \eqref{inequality_lp} leads to the following inequalities
\begin{align*}%\label{inequality_w1p}
    \min \left\{\|u\|_{1,\mathcal{H},0}^p,\|u\|_{1,\mathcal{H},0}^q \right\} \leq \|\nabla u\|_p^p+\|\nabla u\|^q_{q,\mu}
    \leq \max\left \{\|u\|_{1,\mathcal{H},0}^p,\|u\|_{1,\mathcal{H},0}^q\right\}
\end{align*}
for all $u\in W^{1,\mathcal{H}}_0(\Omega)$.

As usual, we denote by $C^1_0(\overline{\Omega})$ the ordered Banach space
\begin{align*}
    C^1_0(\overline{\Omega})= \left\{u \in C^1(\overline{\Omega}): u\big|_{\partial \Omega}=0 \right\},
\end{align*}
with positive cone
\begin{align*}
    C^1_0(\overline{\Omega})_+=\left\{u \in C^1_0(\overline{\Omega}): u(x) \geq 0 \,\,\forall x \in \overline{\Omega}\right\}.
\end{align*}
This cone has a nonempty interior given by
\begin{align*}
    \interior=\left\{u \in C^1_0(\overline{\Omega}): u(x)>0 \,\,\forall x \in  \Omega \text{ and } \frac{\partial u}{\partial n}(x)<0 \,\,\forall x \in  \partial \Omega \right\},
\end{align*}
where $n=n(x)$ is the outer unit normal at $x \in \partial \Omega$.

Next, let us recall some basic facts about the spectrum of the negative Dirichlet $r$-Laplacian with $1<r<\infty$. We consider the problem
\begin{equation}\label{eigenvalue_problem}
    \begin{aligned}
	-\Delta_r u& =\lambda|u|^{r-2}u\quad && \text{in } \Omega,\\
	u
       & = 0  &&\text{on } \partial \Omega.
    \end{aligned}
\end{equation}

A number $\lambda \in \R$ is an eigenvalue of $\left(-\Delta_r,W^{1,r}_0(\Omega)\right)$ if problem \eqref{eigenvalue_problem} has a nontrivial solution $u \in W^{1,r}_0(\Omega)$ which is called an eigenfunction corresponding to the eigenvalue $\lambda$. We denote by $\sigma_r$ the set of eigenvalues of $\left(-\Delta_r,W^{1,r}_0(\Omega)\right)$. We know that the set $\sigma_r$ has a smallest element $\lambda_{1,r}$ which is positive, isolated, simple and it can be variationally characterized through
    \begin{align*}%\label{lambda_one}
	\lambda_{1,r} =\inf \left \{ \frac{\|\nabla u\|_{r}^r}{\|u\|_{r}^r}: u \in W^{1,r}_0(\Omega), u \neq 0 \right \}.
    \end{align*}

We refer to L{\^e} \cite{Le-2006} as a reference for these properties. In what follows we denote by $u_{1,r}$ the $L^r$-normalized (i.e., $\|u_{1,r}\|_{r}=1$) positive eigenfunction corresponding to $\lambda_{1,r}$. The nonlinear regularity theory and the nonlinear maximum principle imply that $u_{1,r} \in \interior$, see Lieberman \cite{Lieberman-1988} and Pucci-Serrin \cite{Pucci-Serrin-2007}.

Let $A\colon \Wpzero{\mathcal{H}}\to \Wpzero{\mathcal{H}}^*$ be the operator defined by
\begin{align}\label{operator_representation}
    \langle A(u),v\rangle_{\mathcal{H}} :=\into \left(|\nabla u|^{p-2}\nabla u+\mu(x)|\nabla u|^{q-2}\nabla u \right)\cdot\nabla v \,dx,
\end{align}
for $u,v\in \Wpzero{\mathcal{H}}$,
where $\langle \cdot,\cdot\rangle_{\mathcal{H}}$ is the duality pairing between $\Wpzero{\mathcal{H}}$ and its dual space $\Wpzero{\mathcal{H}}^*$. The properties of the operator $A\colon \Wpzero{\mathcal{H}}\to \Wpzero{\mathcal{H}}^*$ are summarized in the following proposition, see Liu-Dai \cite{Liu-Dai-2018}.

\begin{proposition}%\label{properties_operator_double_phase}
    The operator $A$ defined by \eqref{operator_representation} is bounded, continuous, monotone (hence maximal monotone) and of type $(\Ss_+)$.
\end{proposition}

The norm of $\R^N$ is denoted by $|\cdot|$ and $\cdot$ stands for the inner product in $\R^N$. For $s \in \R$, we set $s^{\pm}=\max\{\pm s,0\}$ and for $u \in \Wp{\mathcal{H}}$ we define $u^{\pm}(\cdot)=u(\cdot)^{\pm}$. It is well known that
\begin{align*}
    u^{\pm} \in \Wp{\mathcal{H}}, \quad |u|=u^++u^-, \quad u=u^+-u^-.
\end{align*}

%***********************************************************************************************************************************
\section{A priori estimates for double phase problems}\label{section_3}
%***********************************************************************************************************************************

In this section we present a priori bounds for double phase problems under very general conditions. Such a result is of independent interest and can be used for several problems of this type. We consider the problem
\begin{equation}\label{problem2}
    \begin{aligned}
	-\divergenz\left(|\nabla u|^{p-2}\nabla u+\mu(x) |\nabla u|^{q-2}\nabla u\right) & =g(x,u,\nabla u)\quad && \text{in } \Omega,\\
	u & = 0 &&\text{on } \partial \Omega,
    \end{aligned}
\end{equation}
where we suppose the following conditions on the function $g\colon\Omega\times\R\times\R^N\to \R$.
\begin{enumerate}
    \item[H($g$)]
	$g\colon\Omega\times\R\times\R^N\to \R$ is a Carath\'eodory function satisfying
	\begin{align*}
	    |g(x,s,\xi)| \leq c_1|\xi|^{p\frac{r-1}{r}}+c_2|s|^{r-1}+c_3,
	\end{align*}
	for a.\,a.\,$x\in\Omega$, for all $s \in \R$ and for all $\xi \in \R^N$ with positive constants $c_1,c_2, c_3$ and $q<r\leq p^*$ where $p^*$ is the critical exponent stated in \eqref{critical_exponent}.
\end{enumerate}

A function $u\in \Wpzero{\mathcal{H}}$ is called a weak solution of problem \eqref{problem2} if
\begin{align}\label{weak_solution2}
    \into \left(|\nabla u|^{p-2}\nabla u+\mu(x)|\nabla u|^{q-2}\nabla u \right)\cdot\nabla v\,dx=\into g(x,u,\nabla u)v\,dx
\end{align}
is satisfied for all test functions $v \in \Wpzero{\mathcal{H}}$.

The following theorem provides the boundedness of weak solutions of problem \eqref{problem2}.

\begin{theorem}%\label{theorem_apriori}
    Let $1<p<q<N$ and let hypotheses \eqref{condition_poincare} and H($g$) be satisfied. Then, each weak solution of problem \eqref{problem2} belongs to $\Linf$.
\end{theorem}

\begin{proof}
        Let $u\in \Wp{\mathcal{H}}\subseteq \Wp{p}$ be a weak solution of problem \eqref{problem}. Since $u=u^+ -u^-$ we can suppose, without any loss of generality, that $u \geq 0$. Let $h>0$ and set $u_h=\min\{u,h\}$. Then we choose
        $v=uu_h^{\kappa p}$ with $\kappa>0$ as test function in \eqref{weak_solution2}. Note that $\nabla v=\nabla u u_h^{\kappa p}+u\kappa p u_h^{\kappa p-1}\nabla u_h$. This gives
    \begin{align}\label{M1}
      \begin{split}
	&\into |\nabla u|^{p} u_h^{\kappa p}\,dx+ \kappa p\int_\Omega  |\nabla u|^{p-2}\nabla u\cdot \nabla u_h u_h^{\kappa p-1}u\,dx\\
	&\quad+\into \mu(x)|\nabla u|^{q}u_h^{\kappa p}\,dx+ \kappa p\int_\Omega \mu(x) |\nabla u|^{q-2}\nabla u\cdot \nabla u_h u_h^{\kappa p-1}u\,dx\\
	&= \into g(x,u,\nabla u) uu_h^{\kappa p} \,dx.
      \end{split}
    \end{align}
    It is easy to see that the third and fourth integral on the left-hand side of \eqref{M1} are positive. Therefore, we have
    \begin{align*}%\label{M2}
      \begin{split}
	&\into |\nabla u|^{p} u_h^{\kappa p}\,dx+ \kappa p\int_\Omega  |\nabla u|^{p-2}\nabla u\cdot \nabla u_h u_h^{\kappa p-1}u\,dx\\
	&\leq  \into g(x,u,\nabla u) uu_h^{\kappa p} \,dx.
      \end{split}
    \end{align*}
    Now we can proceed exactly as in the proof of Theorem 3.1 of Marino-Winkert \cite[starting with (3.2)]{Marino-Winkert-2019} to obtain that $u \in \Linf$.
\end{proof}

%***********************************************************************************************************************************
\section{Main results}%\label{section_4}
%***********************************************************************************************************************************

We assume the following hypotheses on the nonlinearity $f\colon \Omega \times \R \to \R$.

\begin{enumerate}[leftmargin=1cm]
    \item[H($f$)]
	$f\colon \Omega \times \R \to \R$ is a Carath\'{e}odory function such that
	\begin{enumerate}[leftmargin=0cm]
	    \item[(i)]		
                $f$ is bounded on bounded sets;
	    \item[(ii)]
	      \begin{align*}
		  \lim_{s \to \pm \infty}\,\frac{f(x,s)}{|s|^{q-2}s}=+\infty \quad\text{uniformly for a.\,a.\,}x\in\Omega;
	      \end{align*}
	      \item[(iii)]
	      \begin{align*}
		  \lim_{s \to 0}\,\frac{f(x,s)}{|s|^{p-2}s}=0 \quad\text{uniformly for a.\,a.\,}x\in\Omega.
	      \end{align*}
        \end{enumerate}
\end{enumerate}

\begin{example}
        The following functions $f_i\colon\Omega \times\R\to\R$ ($i=1,2,3$) satisfy hypotheses H($f$):
        \begin{align*}
                f_1(x,s)&=a(x)|s|^{r-2}s,\\
                f_2(x,s)&=a(x)|s|^{q-2}s\ln(1+|s|),\\
                f_3(x,s)&=
                \begin{cases}
                        |s|^{q-2}se^{-s-1} &\text{if } s<-1,\\
                        \frac{|s|^p}{2} ((s-1)\cos(s+1)+s+1)&\text{if } -1 \leq s \leq 1,\\
                        (1+(s-1)|x|)s^{q-1}e^{s-1} & \text{if } s>1,
                \end{cases}              
        \end{align*}
	where $a\in \Linf$ and $q<r<\infty$.
\end{example}

We say that $u\in \Wpzero{\mathcal{H}}$ is a weak solution of problem \eqref{problem} if it satisfies
\begin{align*}%\label{weak_solution}
    \into \left(|\nabla u|^{p-2}\nabla u+\mu(x)|\nabla u|^{q-2}\nabla u \right)\cdot\nabla v \,dx=\into\left(\lambda|u|^{p-2}u-f(x,u)\right) v \,dx
\end{align*}
for all test functions $v \in \Wpzero{\mathcal{H}}$.

Our main existence result reads as follows.

\begin{theorem}%\label{theorem_existence}
    Let $1<p<q<N$, let hypotheses \eqref{condition_poincare} and H($f$) be satisfied and assume that $\lambda>\lambda_{1,p}$. Then problem \eqref{problem} admits at least two nontrivial weak solutions $u_+,u_- \in \Wpzero{\mathcal{H}}\cap \Linf$ such that $u_+\geq 0$ and $u_- \leq 0$ in $\Omega$.
\end{theorem}

\begin{proof}
    We first note that for each $a>0$, by applying H($f$)(ii), there exists a constant $M=M(a)>1$ such that
    \begin{align}\label{1}
	f(x,s)s \geq a |s|^{q}\quad\text{for a.\,a.\,}x\in\Omega \text{ and for all }|s|\geq M.
    \end{align}
    Taking $a=\lambda$ and a constant function $\overline{u}\in[M,\infty)$ and using \eqref{1}, $p<q$ and $M>1$, we have that
    \begin{align}\label{2}
	0 \geq \lambda \overline{u}^{p-1}-f(x,\overline{u}) \quad\text{for a.\,a.\,}x\in \Omega.
    \end{align}
    Now we introduce the truncation function $h_+\colon \Omega\times\R\to\R$ defined by
    \begin{align}\label{3}
	h_+(x,s)=
	\begin{cases}
	    0 &\text{if } s<0,\\
	    \lambda s^{p-1} -f(x,s) &\text{if }0\leq s \leq \overline{u},\\
	    \lambda \overline{u}^{p-1} -f(x,\overline{u})&\text{if }\overline{u}<s.
	\end{cases}
    \end{align}
    It is clear that $h_+$ is a Carath\'{e}odory function. We set $H_+(x,s)=\int^s_0 h_+(x,t)\,dt$ and consider the $C^1$-functional $\ph_+\colon \Wpzero{\mathcal{H}}\to\R$ given by
    \begin{align*}
	\ph_+(u)=\into \left[\frac{1}{p}|\nabla u|^p+\frac{\mu(x)}{q}|\nabla u|^q \right]\,dx -\into H_+(x,u)\,dx.
    \end{align*}
    Due to the truncation operator we see that the functional $\ph_+$ is coercive and since the embedding $\Wpzero{\mathcal{H}}\hookrightarrow \Lp{r}$ is compact, it is also sequentially weakly lower semicontinuous. Therefore, its global minimizer $u_+ \in \Wpzero{\mathcal{H}}$ exists, that is,
    \begin{align*}%\label{4}
	\ph_+(u_+)=\inf \left[ \ph_+(u)\colon u\in \Wpzero{\mathcal{H}}\right].
    \end{align*}
    Let us prove that this minimizer is nontrivial. To this end, due to H($f$)(iii), we find for each $\eps>0$ a number $\delta=\delta(\eps) \in (0,\overline{u})$ such that
    \begin{align}\label{5}
	F(x,s)\leq \frac{\eps}{p}|s|^p\quad\text{for a.\,a.\,}x\in \Omega \text{ and for all }|s|\leq \delta,
    \end{align}
    where $F(x,s)=\int_0^sf(x,t)\,dt$. We recall that $u_{1,p}\in \interior$ and so we can choose $t\in (0,1)$ small enough such that $tu_{1,p}(x) \in [0,\delta]$ for all $x\in\close$. Moreover, we know that $\|u_{1,p}\|_p=1$.

    Using this fact, \eqref{3}, \eqref{5} and the fact that $\delta<\overline{u}$ we obtain
    \begin{align}\label{6}
      \begin{split}
	\ph_+\left(t u_{1,p}\right)
	& =\into \left[\frac{1}{p}|\nabla (t u_{1,p})|^p+\frac{\mu(x)}{q}|\nabla (t u_{1,p})|^q \right]\,dx -\into H_+(x,t u_{1,p})\,dx\\
	& \leq \frac{t^p}{p}\left\|\nabla u_{1,p}\right\|_p^p+\frac{t^q}{q}\left\|\nabla u_{1,p}\right\|_{q,\mu}^q-\frac{\lambda t^p}{p}+\frac{\eps t^p}{p}\\
	&\leq  \frac{t^p \lambda_{1,p}}{p}+\frac{t^q}{q} \left\|\nabla u_{1,p}\right\|_{q,\mu}^q -\frac{\lambda t^p}{p}+\frac{\eps t^p}{p}\\
	& = t^p \left[\frac{\lambda_{1,p}-\lambda +\eps}{p} \right]+\frac{t^q}{q} \left\|\nabla u_{1,p}\right\|_{q,\mu}^q.
      \end{split}
    \end{align}
    Since $\lambda > \lambda_{1,p}$ we can choose $\eps \in (0, \lambda-\lambda_{1,p})$ in \eqref{6} which shows that (note that $p<q$)
    \begin{align*}
	\ph_+(tu_{1,p})<0=\ph_+(0) \quad\text{for sufficiently small }t>0.
    \end{align*}
    Therefore, $u_+ \neq 0$.

    Since $u_+$ is a global minimizer of $\ph_+$ we have $(\ph_+)'(u_+)=0$, that is,
    \begin{align}\label{7}
      \begin{split}
	& \into \left(|\nabla u_+|^{p-2}\nabla u_++\mu(x)|\nabla u_+|^{q-2}\nabla u_+ \right)\cdot\nabla v \,dx =\into h_+(x,u_+)v \,dx
      \end{split}
    \end{align}
    for all $v \in \Wpzero{\mathcal{H}}$. Taking $v=-(u_+)^-$ as test function in \eqref{7} gives $(u_+)^-=0$ and so $u_+\geq 0$. Next, we choose $v=(u_+-\overline{u})^+ \in \Wpzero{\mathcal{H}}$ as test function in \eqref{7}. This leads to
    \begin{align*}%\label{7}
      \begin{split}
	& \into \left(|\nabla u_+|^{p-2}\nabla u_++\mu(x)|\nabla u_+|^{q-2}\nabla u_+ \right)\cdot\nabla (u_+-\overline{u})^+ \,dx\\
	& =\into h_+(x,u_+)(u_+-\overline{u})^+ \,dx\\
	& =\into\left( \lambda \overline{u}^{p-1}-f(x, \overline{u})\right)(u_+-\overline{u})^+ \,dx\\
	& \leq 0,
      \end{split}
    \end{align*}
    by \eqref{2}. Hence, from the last inequality we derive
    \begin{align*}%\label{7}
      \begin{split}
	& \int_{\{u_+>\overline{u}\}} |\nabla u_+|^{p}\,dx+\int_{\{u_+>\overline{u}\}}\mu(x)|\nabla u_+|^{q}\,dx
	\leq 0.
      \end{split}
    \end{align*}
    Therefore, $u_+\leq \overline{u}$. Then, with view to the truncation function defined in \eqref{3}, we see that $u_+ \in \Wpzero{\mathcal{H}}\cap \Linf$ is a weak solution of our original problem. 

    The proof for the negative solution works in a similar way. Indeed, as a conclusion of \eqref{1}, we can choose $a=\lambda$ and $\underline{v} \in (-\infty, -M]$ such that
    \begin{align*}%\label{corollary_A_sub}
       0 \leq \lambda |\underline{v}|^{p-2}\underline{v}-f(x,\underline{v}) \quad \text{for a.a.\,}x \in \Omega.%N_f(\underline{v}) \quad \text{in } W^{-1,p'}(\Omega).
    \end{align*}
    Then, we introduce the truncation $h_-\colon \Omega\times\R\to\R$ defined by
    \begin{align*}%\label{truncation_two}
	h_-(x,s)=
	\begin{cases}
	    \lambda |\underline{v}|^{p-2}\underline{v}-f(x,\underline{v}) \qquad & \text{if } s<\underline{v},\\
	    \lambda |s|^{p-2}s-f(x,s) & \text{if }\underline{v} \leq s \leq 0,\\
	    0 & \text{if } 0<s,
	\end{cases}
    \end{align*}
    and the $C^1$-functional $\varphi_{-}\colon \Wpzero{\mathcal{H}} \to \R$ defined by
    \begin{align*}
	\ph_-(u)=\into \left[\frac{1}{p}|\nabla u|^p+\frac{\mu(x)}{q}|\nabla u|^q \right]\,dx -\into H_-(x,u) \,dx,
    \end{align*}
    where $H_-(x,s)=\int_0^s h_-(x,t)dt$. In the same way we prove that $\ph_-$ has a global minimizer $u_-$ which turns out to be a nontrivial, negative, bounded weak solution of \eqref{problem}.
\end{proof}

\begin{remark}
        It would also be interesting to obtain sign-changing solutions for problem \eqref{problem}. As for now, the methods based on extremal solutions cannot be applied due to the lack of regularity results for double phase problems. Moreover, the application of the Nehari manifold method does not seem to work for our problem because of the sign of the nonlinearity $f$, 
         compare with the assumption in Gasi\'nski-Papageorgiou \cite{Gasinski-Papageorgiou-2019} and Liu-Dai \cite{Liu-Dai-2018}.
\end{remark}

\section*{Acknowledgment}

The authors wish to thank knowledgeable referees for their corrections and remarks.


\begin{thebibliography}{99}

\bibitem{Bahrouni-Radulescu-Repovs-2019}
    A.\,Bahrouni, V.\,D.\,R\u{a}dulescu, D.\,D.\,Repov\v{s},
    {\it Double phase transonic flow problems with variable growth: nonlinear patterns and stationary waves},
    Nonlinearity {\bf 32} (2019), no.\,7, 2481--2495.

\bibitem{Baroni-Colombo-Mingione-2015}
    P.\,Baroni, M.\,Colombo, G.\,Mingione,
    {\it Harnack inequalities for double phase functionals},
    Nonlinear Anal.\,{\bf 121} (2015), 206--222.
    
\bibitem{Baroni-Colombo-Mingione-2016}
    P.\,Baroni, M.\,Colombo, G.\,Mingione,
    {\it Non-autonomous functionals, borderline cases and related function classes},
    St.\,Petersburg Math.\,J.\,{\bf 27} (2016), 347--379.

\bibitem{Baroni-Kussi-Mingione-2015}
    P.\,Baroni, T.\,Kuusi, G.\,Mingione,
    {\it Borderline gradient continuity of minima},
    J.\,Fixed Point Theory Appl.\,{\bf 15} (2014), no.\,2, 537--575.

\bibitem{Baroni-Colombo-Mingione-2018}
    P.\,Baroni, M.\,Colombo, G.\,Mingione,
    {\it Regularity for general functionals with double phase},
    Calc.\,Var.\,Partial Differential Equations {\bf 57} (2018), no.\,2, Art.\,62, 48 pp.

\bibitem{Cencelj-Radulescu-Repovs-2018}
    M.\,Cencelj, V.\,D.\,R\u{a}dulescu, D.\,D.\,Repov\v{s},
    {\it Double phase problems with variable growth},
    Nonlinear Anal.\,{\bf 177} (2018), part A, 270--287.
    
\bibitem{Colasuonno-Squassina-2016}
    F.\,Colasuonno, M.\,Squassina,
    {\it Eigenvalues for double phase variational integrals},
    Ann.\,Mat. Pura Appl.\,(4) {\bf 195} (2016), no.\,6, 1917--1959.

\bibitem{Colombo-Mingione-2015a}
    M.\,Colombo, G.\,Mingione,
    {Bounded minimisers of double phase variational integrals},
    Arch. Ration.\,Mech.\,Anal.\,{\bf 218} (2015), no.\,1, 219--273.

\bibitem{Colombo-Mingione-2015b}
    M.\,Colombo, G.\,Mingione,
    {\it Regularity for double phase variational problems},
    Arch.\,Ration. Mech.\,Anal.\,{\bf 215} (2015), no.\,2, 443--496.

\bibitem{Cupini-Marcellini-Mascolo-2015}
    G.\,Cupini, P.\,Marcellini, E.\,Mascolo,
    {\it Local boundedness of minimizers with limit growth conditions},
    J.\,Optim.\,Theory Appl.\,{\bf 166} (2015), no.\,1, 1--22.

\bibitem{Gasinski-Papageorgiou-2019}
    L.\,Gasi\'nski, N,\,S.\,Papageorgiou,
    {\it Constant sign and nodal solutions for superlinear double phase problems},
    Adv. Calc. Var., https://doi.org/10.1515/acv-2019-0040.

\bibitem{Gasinski-Winkert-2019}
    L.\,Gasi\'nski, P.\,Winkert,
    {\it Existence and uniqueness results for double phase problems with convection term},
    J. Differential Equations, https://doi.org/10.1016/j.jde.2019.10.022.

\bibitem{Le-2006}
    A.\,L{\^e},
    {\it Eigenvalue problems for the {$p$}-{L}aplacian},
    Nonlinear Anal.\,{\bf 64} (2006), no.\,5, 1057--1099.

\bibitem{Lieberman-1988}
    G.\,M.\,Lieberman,
    {\it Boundary regularity for solutions of degenerate elliptic equations},
    Nonlinear Anal.\,{\bf 12} (1988), no.\,11, 1203--1219.

\bibitem{Liu-Dai-2018}
    W.\,Liu, G.\,Dai,
    {\it Existence and multiplicity results for double phase problem},
    J.\,Differential Equations {\bf 265} (2018), no.\,9, 4311--4334.

\bibitem{Marcellini-1989}
    P.\,Marcellini,
    {The stored-energy for some discontinuous deformations in nonlinear elasticity}, in ``Partial differential equations and the calculus of variations, {V}ol. {II}'', vol.\,2, 767--786, Birkh\"{a}user Boston, Boston, 1989.

\bibitem{Marcellini-1991}
    P.\,Marcellini,
    {\it Regularity and existence of solutions of elliptic equations with {$p,q$}-growth conditions},
    J.\,Differential Equations {\bf 90} (1991), no.\,1, 1--30.

\bibitem{Marino-Winkert-2019}
    G.\,Marino, P.\,Winkert,
    {\it Moser iteration applied to elliptic equations with critical growth on the boundary},
    Nonlinear Anal.\,{\bf 180} (2019), 154--169.

\bibitem{Papageorgiou-Radulescu-Repovs-2019a}
    N.\,S.\,Papageorgiou, V.\,D.\,R\u{a}dulescu, D.\,D.\,Repov\v{s},
    {\it Double-phase problems and a discontinuity property of the spectrum},
    Proc.\,Amer.\,Math.\,Soc.\,{\bf 147} (2019), no.\,7, 2899--2910.
    
\bibitem{Papageorgiou-Radulescu-Repovs-2019b}
    N.\,S.\,Papageorgiou, V.\,D.\,R\u{a}dulescu, D.\,D.\,Repov\v{s},
    {\it Double-phase problems with reaction of arbitrary growth},
    Z.\,Angew.\,Math.\,Phys.\,{\bf 69} (2018), no.\,4, Art.\,108, 21 pp.
         
\bibitem{Papageorgiou-Scapellato-2020}
    N.\,S.\,Papageorgiou, A.\,Scapellato,
    {\it Constant sign and nodal solutions for parametric {$(p,2)$}-equations},
    Adv.\,Nonlinear Anal.\,{\bf 9} (2020), no.\,1, 449--478.
    
\bibitem{Perera-Squassina-2019}
    K.\,Perera, M.\,Squassina,
    {\it Existence results for double-phase problems via Morse theory},
    Commun.\,Contemp.\,Math.\,{\bf 20} (2018), no.\,2, 1750023, 14 pp.

\bibitem{Pucci-Serrin-2007}
   P.\,Pucci, J.\,Serrin,
   ``The Maximum Principle'',
   Birkh\"auser Verlag, Basel, 2007.

\bibitem{Radulescu-2019}
    V.\,D.\,R\u{a}dulescu,
    {\it Isotropic and anistropic double-phase problems: old and new},
    Opuscula Math.\,{\bf 39} (2019), no.\,2, 259--279.
   
\bibitem{Zhang-Radulescu-2018}
    Q.\,Zhang,V.\,D.\,R\u{a}dulescu,
    {\it Double phase anisotropic variational problems and combined  effects of reaction and absorption terms},
    J.\,Math.\,Pures Appl.\,(9) {\bf 118} (2018), 159--203.
    
\bibitem{Zhikov-1986}
    V.\,V.\,Zhikov,
    {\it Averaging of functionals of the calculus of variations and elasticity theory},
    Izv.\,Akad.\,Nauk SSSR Ser.\,Mat.\,{\bf 50} (1986), no.\,4, 675--710.

\bibitem{Zhikov-1995}
    V.\,V.\,Zhikov,
    {\it On {L}avrentiev's phenomenon},
    Russian J.\,Math.\,Phys.\,{\bf 3} (1995), no.\,2, 249--269.

\bibitem{Zhikov-1997}
    V.\,V.\,Zhikov,
    {\it On some variational problems},
    Russian J.\,Math.\,Phys.\,{\bf 5} (1997), no.\,1, 105--116.

\bibitem{Zhikov-Kozlov-Oleinik-1994}
    V.\,V.\,Zhikov, S.\,M.\,Kozlov, O.\,A.\,Ole\u{\i}nik,
    ``Homogenization of Differential Operators and Integral Functionals,
    Springer-Verlag, Berlin, 1994.

\end{thebibliography}
\end{document}